\documentclass[preprint]{imsart}
\usepackage[utf8]{inputenc}

\usepackage{amsmath,amssymb,amsthm}
\usepackage{mathrsfs}
\usepackage[english]{babel}
\usepackage{color}
\usepackage{graphicx}
\usepackage{color}
\usepackage{indentfirst}
\usepackage{multirow}
\usepackage{float}
\usepackage{subfig}

\usepackage[authoryear]{natbib}

\startlocaldefs
\theoremstyle{plain}
\newtheorem{theorem}{Theorem}
\def\@journal{Submitted}
\endlocaldefs

\begin{document}

\newcommand\cov{\mathop{\text{Cov}}}
\newcommand\tr{\mathop{\text{tr}}}
\newcommand\topii{{\frac1{2\pi i}}}
\newcommand\mbar{\underline{m}}

\newcommand\bbx{{\mathbf x}}
\newcommand\bbu{{\mathbf u}}
\newcommand\bleu[1]{{\color{black}#1}}
\newcommand\reu[1]{{\color{black}#1}}
\renewcommand{\theenumi}{(\roman{enumi})}

\newcommand{\mydp}[1]{\displaystyle{#1}}
\newcommand{\bs}[1]{\boldsymbol{#1}}
\def\widebar#1{\overline{#1}}

\newtheorem{lemma}{Lemma}
\newtheorem{cor}{Corollary}
\newtheorem{thm}{Theorem}
\newtheorem{prop}{Proposition}
\renewcommand{\theequation}{\thesection.\arabic{equation}}
\renewcommand{\thelemma}{\thesection.\arabic{lemma}}
\renewcommand{\theprop}{\thesection.\arabic{prop}}
\renewcommand{\thethm}{\thesection.\arabic{thm}}
\renewcommand{\thecor}{\thesection.\arabic{cor}}
\newcommand\newsec{%
\setcounter{equation}{0}%
\setcounter{lemma}{0}%
\setcounter{prop}{0}%
\setcounter{thm}{0}%
\setcounter{cor}{0}%
}%

\renewcommand{\theequation}{\thesection.\arabic{equation}}

\newcommand{\al}{\alpha}
\newcommand{\om}{\omega}
\newcommand{\Om}{\Omega}
\newcommand{\La}{\Lambda}
\newcommand{\la}{\lambda}
\newcommand{\De}{\Delta}
\newcommand{\de}{\delta}
\newcommand{\ep}{\epsilon}
\newcommand{\be}{\beta}
\newcommand{\ga}{\gamma}
\newcommand{\Ga}{\Gamma}
\newcommand{\te}{\theta}
\newcommand{\Te}{\Theta}
\newcommand{\si}{\sigma}
\newcommand{\Si}{\Sigma}
\newcommand{\vep}{\varepsilon}
\newcommand{\veps}{\varepsilon}
\newcommand{\eps}{\epsilon}
\newcommand{\ze}{\zeta}
\newcommand{\T}{\mathrm{\tiny T}} 
\newcommand{\pa}{\partial}
\newcommand{\dd}[2]{\frac{\pa #1}{\pa #2}} 

\newcommand{\Z}{\mathbb{Z}}
\newcommand{\R}{\mathbb{R}}
\newcommand{\CC}{\mathbb{C}}   
\newcommand{\Q}{\mathbb{Q}}
\newcommand{\N}{\mathbb{N}}

\def\cov{\mathop{\hbox{\rm cov}}}
\def\var{\mathop{\hbox{\rm var}}}

\newcommand{\PP}{\; \mathbb{P}}
\newcommand{\E}{\;  \mathbb{E}}

\newcommand{\rw}{\rightarrow}
\newcommand{\lrw}{\longrightarrow}
\newcommand{\Rw}{\Rightarrow}
\newcommand{\Lrw}{\Longrightarrow}


\def\minmax{\mathop{\hbox{\rm min\,max}}}
\def\tr{\mathop{\hbox{\rm tr}}}
\def\ln{\mathop{\hbox{\rm $\ell n$}}}
\def\Arg{\mathop{\hbox{\rm Arg}}}
\def\diag{\mathop{\hbox{\rm diag}}}
\def\grad{\mathop{\nabla}}
\def\card{\mathop{\hbox{\rm card}}}
\def\bc{\mathop{\hbox{\mbox{\large\rm $|$}}}}
\def\telque{\; : \;} 
\def\pds#1#2{{\langle #1,#2 \rangle}}         
\def\bpv{\mathop{;}}                          
\def\dfrac#1#2{\frac{\displaystyle #1}{\displaystyle #2}}
\def\ffrac#1{\frac{1}{#1}}
\newcommand{\norm}[2]{ \| #1 \|_{#2} }
\def\un{\mbox{\large 1\hskip-0.30em I}} 
\def\indic#1{\un_{#1}}                        
\def\cf{{\em cf. }} 
\def\sgn{\mathop{\hbox{\rm sgn}}}
\newcommand{\rmx}[1]{{\mbox{\rm #1}}}
\newcommand{\rmxs}[1]{{\mbox{\rm {\scriptsize #1}}}}
\newcommand{\rmxt}[1]{{\mbox{\rm {\tiny #1}}}}
\newcommand{\trsp}[1]{{{#1}^\T}}                 
\newcommand{\fq}[2]{\trsp{#1} #2 #1}

\newcommand{\eq}[1]{(\ref{#1})}         
\newcommand{\cte}{\rmx{cte }}                

\def\slabel#1{\label{#1} \hfill \hskip2mm\hbox to 4mm{\dotfill}%
        \mbox{\tiny\em(#1)}\hbox to 4mm{\dotfill}}

\def\ddemo#1{\noindent{\textbf{Proof {#1}. \quad}}}
\def\findemo{\hskip3mm\mbox{\quad\vbox{\hrule height 3pt depth 6pt width 6pt}}}

\begin{frontmatter}
  \title{A note on the CLT of the LSS for sample covariance matrix from a spiked population model}

\thankstext{T1}{Research of this author was partly supported by
 the National Natural Science Foundation of China (Grant No. 11071213), the Natural Science
Foundation of Zhejiang Province (No. R6090034), and the Doctoral Program Fund of
Ministry of Education (No. J20110031).}
\thankstext{T2}{Research of this author was partly supported by the U.S. Army Research Office under Grant W911NF-09-1-0266.}
\thankstext{T3}{Research of this author was partly supported by
 a HKU start-up grant.}
  \begin{aug}
    \author{\fnms{Qinwen} \snm{Wang}\thanksref{T1}\ead[label=e1]{wqw8813@gmail.com}}
    \and
    \author{\fnms{Jack W.} \snm{Silverstein}\thanksref{T2}\ead[label=e2]{jack@ncsu.edu}}
    \and
    \author{\fnms{Jian-feng} \snm{Yao}\thanksref{T3}\ead[label=e3]{jeffyao@hku.hk}}

    \runauthor{Q. Wang, J. Silverstein and J. Yao}

    \affiliation{Zhejiang  University North Carolina State University and The University of Hong Kong}

    \address{Qinwen Wang  \\
      Department of Mathematics\\
      Zhejiang University \\
      \printead{e1}
    }

    \address{Jack W. Silverstein  \\
      Department of Mathematics\\
      North Carolina State University\\
      \printead{e2}
    }

    \address{Jianfeng Yao \\
      Department of Statistics and Actuarial Science\\
      The University of Hong Kong\\
      Pokfulam, \quad
      Hong Kong \\
      \printead{e3}
    }
  \end{aug}

  \begin{abstract}
    In this note, we establish an asymptotic expansion for the
  centering parameter appearing in the central limit theorems
  for  linear spectral statistic of large-dimensional
  sample covariance matrices when the population has a spiked
  covariance
  structure.
  As an application, we provide an asymptotic power function for
  the corrected likelihood ratio \bleu{statistic}
  for testing the presence of spike eigenvalues in the population
  covariance matrix. This result generalizes \bleu{an existing  formula
  from the literature} where  only one simple spike exists.
  \end{abstract}


\end{frontmatter}

\section{Introduction}

Let $(\Sigma_p)$ be a sequence of $p\times p$ non-random and nonnegative
definite Hermitian matrices and let $(w_{ij})$, $i,j\ge 1$ be a
doubly infinite array of i.i.d. complex-valued random variables
satisfying
\[  \label{eq:moments}
  \E(w_{11})=0,~~
  \E(|w_{11}|^2)=1,~~
  \E(|w_{11}|^4)<\infty.
\]
Write $Z_n=(w_{ij})_{1\le i\le p,1\le j \le n}$, the upper-left
$p\times n$ block,  where $p=p(n)$ is
related to $n$ such that when $n\rw \infty$, $p/n\rw y >0$.
 Then the matrix
$S_n=\frac1n \Sigma_p^{1/2} Z_nZ_n^*\Sigma_p^{1/2}$ can be considered as the
sample covariance matrix of an i.i.d. sample
$(\bbx_1,\ldots,\bbx_n)$ of $p$-dimensional observation vectors
$\bbx_j=\Sigma_p^{1/2}{\bbu_j}$ where ${\bbu_j}=(w_{ij})_{1\le i\le p}$
denotes the $j$-th column of $Z_n$. Note that for any \bleu{nonnegative}
definite Hermitian matrix $A$,
$A^{1/2}$ denotes a Hermitian square root and we call the {\em spectral
  distribution} (SD) the distribution generated by its eigenvalues.

Assume that the SD $H_n$
of $\Sigma_p$ converges weakly to a nonrandom
probability distribution $H$ on $[0,\infty)$. It is then well-known
that the SD $F^{S_n}$ of $S_n$, generated by its eigenvalues $\la_{n,1} \ge \cdots\ge\la_{n,p}$,
converges to a nonrandom limiting SD  $G$ \citep{MP,Silverstein95}.
The so-called {\em null case} corresponds to the situation $\Sigma_p\equiv
I_p$, so $H_n\equiv\delta_1$ and the limiting SD is the seminal
 Mar\v{c}enko-Pastur law $G^y$ with index $y$ and
 support $[a_y,b_y]$ where $ a_y=(1-\sqrt y)^2$,
$b_y=(1+\sqrt y)^2$, and an additional mass at the origin if $y>1$.

In this paper we consider the
{\em spiked population model} introduced in
\citet{John01} where
the eigenvalues of $\Sigma_p$ are
\begin{equation}
  \label{model}
  \underbrace{a_1,\cdots,a_1}_{n_1},
  \ldots,
  \underbrace{a_k,\cdots,a_k}_{n_k},
  \underbrace{ 1,\cdots,1}_{p-M}.
\end{equation}
Here $M$   and  the multiplicity numbers $(n_k)$ are fixed and
satisfy $n_1+\cdots+n_k=M$.
In other words,
all the population eigenvalues are unit except
some fixed  number of them (the spikes).
The  model
can be viewed as a {finite-rank  perturbation of the
null case}.
Obviously, the limiting SD $G$ of  $S_n$ is not affected by this
perturbation.
However,  the asymptotic behaviour of the extreme eigenvalues of
$S_n$ is significantly different from  the  {null case}.
The analysis of this new behaviour of extreme eigenvalues has been
an active area in the last few years, see e.g.
\citet{BBP05}, \citet{Baik06}, \citet{Paul07},  \citet{BaiYao08},
\cite{Benaych11},
 \citet{NS},
\cite{BenaychNadakuditi11} and
\citet{BY12}. In particular, the base component of the population SD $H_n$
in the last three references has been extended to a form
more general than the simple Dirac mass $\delta_1$ of
the null case.

For statistical applications, besides the principal components
analysis which is
indeed the origin of spiked models (\cite{John01}),
large-dimensional strict factor models are equivalent to a spiked
population model and can be analyzed using the above-mentioned
results. Related recent contributions in the area include,
among others,
\citet{KN08,KN09},
\citet{Onatski09,Onatski10,Onatski12}
 and
\citet{PY12} and \bleu{they all}   concern the problem
of estimation  and testing
the number of factors (or spikes).

In this note, we analyze  the effects caused by the spike
eigenvalues on the fluctuations of  linear spectral statistics of the form
\begin{equation}\label{stat}
  T_n(f) = \sum_{i=1}^p  f(\la_{n,i}) = F^{S_n}(f)~,
\end{equation}
where $f$ is a given function.
Similarly to the convergence of the SD's,
the presence of the spikes
does not prevent   a  central limit theorem for $T_n(f)$;
however as we will see, the centering term in the CLT will be modified
according to the values of the spikes.
As this  term has no
explicit form, our main result
is  an asymptotic
expansion presented in Section~\ref{main}.
To illustrate the  importance of such expansions, we present
in Section \ref{application}
an application for the determination of the power function
for testing the presence of spikes.
The Appendix contains some technical derivations.

\section{Centering parameter in the CLT of the LSS from a spiked population model}
\label{main}

Fluctuations of linear spectral statistics of form \eqref{stat}
are indeed covered by a central limit theory \bleu{initiated}  in
\citet{BS04}.
The theory was later improved by \citet{PanZhou08} where
the  restriction  $E(|w_{11}|^4)=3$
matching the real Gaussian case  was removed.

Let $f_1,\ldots,f_L$ be $L$ functions analytic on an open domain of
the complex plan  including the
support of the limiting SD.
These central limit theorems state that the random vector
\[
(X_n(f_1), \cdots, X_n(f_L))~,
\]
where
\[
X_n(f)= p \left[F^{S_n}(f) -F^{y_n, H_n}(f)\right]=  p\int f(x)d(F^{S_n}-F^{y_n, H_n})(x)~,
\]
converges  weakly to a Gaussian vector
\[
(X_{f_1}, \cdots, X_{f_L})
\]
with known mean function $E[X_f]$ and covariance function
$Cov(X_f, X_g)$  that can
be calculated from  contour integrals involving parameters $\underline{m}(z)$
and $H$, where $\underline{m}(z)$ is the companion Stieltjes transform
corresponding to the limiting SD of
$\underline{S}_n=\frac{1}{n}Z_n^{*}\Sigma_p Z_n$.
If the population has a spiked covariance structure,
we know that the limit $H$ and $\underline{m}(z)$ remain
the same  as
the non-spiked case, so the limiting  parameters
$E[X_f]$ and $Cov(X_f, X_g)$ are also unchanged.

It is remarked that the  centering parameter $p F^{y_n,  H_n}(f)$ depends on a particular distribution
$F^{y_n, H_n}$ which is a finite-horizon proxy for
the
limiting SD of $S_n$. The difficulty is that  $F^{y_n, H_n}$ has no
explicit form;
it is indeed {\em implicitly}
defined through
$\underline{m}_n(z)$ (the finite counterpart of $\underline{m}(z)$), which solves the equation:
\begin{equation}
  \label{MPequation}
  z = -\frac{1}{\underline{m}_n} + y_n \int\frac{t}{1+t\underline{m}_n}dH_n(t)~.
\end{equation}
This distribution
depends on the SD $H_n$ which in turn depends on the spike eigenvalues.

More precisely, the SD
 $H_n$ of $\Sigma_p$ is
\begin{equation}
  \label{Hn}
  H_n=\frac{p-M}{p}\delta_1 +\frac{1}{p}\sum_{i=1}^{k}n_i\delta_{a_i}~.
\end{equation}
The term
\[
\frac{1}{p}\sum_{i=1}^{k}n_i\delta_{a_i}
\]
vanishes when $p$ tends to infinity, so it has  no influence
when considering  limiting
spectral distributions. However for the CLT,
the term $p F^{y_n, H_n} (f)$ has a $p$ in front,
and
$\frac{1}{p}\sum_{i=1}^{k}n_i\delta_{a_i}$ times $p$ is of order
$O(1)$, thus cannot be neglected.

It is here reminded that, following \citet{Baik06}, for a \reu{{\em distant
spike}} $a_i$ such that $|a_i-1|>\sqrt{y}$, the corresponding sample
eigenvalue is equal to
$\phi(a_i)=a_i+\frac{ya_i}{a_i-1}$, while for a \reu{{\em close spike}} such that $|a_i-1|\leq \sqrt{y}$, the corresponding sample eigenvalue tends to the edge points $a_y$ and $b_y$.

Our main result is an asymptotic expansion for this centering
parameter.

\begin{theorem}\label{formula}
  Suppose the population has a spiked population structure as stated in \eqref{model} \reu{with $k_1$ distant spikes and $k-k_1$ close spikes (arranged in decreasing order)},
Let $f$ be any analytic function on an open
domain including the support of M-P distribution $G^y$ and all the $\phi(a_i)$, $i\leq k_1$. We have:
\begin{align}
  &F^{y_n, H_n}(f)\nonumber\\
  =&-\frac{1}{2\pi i
  p}\oint_{\mathcal{C}_{1}}f(-\frac{1}{\underline{m}}+\frac{y_n}{1+\underline{m}})(\frac{M}{y_n
  \underline{m}}-\sum_{i=1}^{k}\frac{n_i a_i^2\underline{m}}{(1+a_i \underline{m})^2})d\underline{m}\label{th1}\\
+&\frac{1}{2\pi i
  p}\oint_{\mathcal{C}_{1}}f^{'}(-\frac{1}{\underline{m}}+\frac{y_n}{1+\underline{m}})\sum_{i=1}^{k}\frac{(1-a_i)n_i}{(1+a_i
  \underline{m})(1+\underline{m})}(\frac{1}{\underline{m}}-\frac{y_n \underline{m}}{(1+\underline{m})^2})d\underline{m} \label{th2}\\
+&(1-\frac{M}{p}) G^{y_n} (f)+\frac{1}{p}\sum_{i=1}^{k_1}n_if(\phi(a_i))+O(\frac{1}{n^2})\label{th3}~;
\end{align}
Here  $\underline{m}=\underline{m}_n$ is the \bleu{companion}
Stieltjes transform of $F^{y_n,H_n}$ defined in \eqref{MPequation}, $G^{y_n}(f)$ is the integral of $f$ with respect to the Mar\v{c}enko-Pastur distribution
with index $y_n=p/n$. And
\reu{
 \begin{enumerate}
   \item when $0<y_n<1$, the first $k_1$ spike eigenvalues $a_i 's$ satisfy $|a_i-1|>\sqrt{y_n}$, the remaining $k-k_1$ satisfy $|a_i-1|\leq\sqrt{y_n}$, $\mathcal{C}_{1}$ is a  contour counterclockwise, when restricted to the real axes, encloses the interval $[\frac{-1}{1-\sqrt{y_n}},\frac{-1}{1+\sqrt{y_n}}]$;
   \item when $y_n\geq 1$, the first $k_1$ spike eigenvalues $a_i 's$ satisfy $a_i-1>\sqrt{y_n}$, the remaining $k-k_1$ satisfy $0<a_i\leq1+\sqrt{y_n}$,  $\mathcal{C}_{1}$ is a contour clockwise, when restricted to the real axes, encloses the interval $[-1,\frac{-1}{1+\sqrt{y_n}}]$.
 \end{enumerate}
 }
\noindent If there are no distant spikes then the second term in (2.7) does not appear.
\end{theorem}

\begin{proof}

We divide the proof  into three parts according to whether $0<y_n<1$, $y_n>1$ or $y_n=1$.
\bigskip

\noindent\reu{\underline{Case of $0<y_n<1$:}}

Recall that $G^{y_n}(f)=\int f(x)d G^{y_n}(x)$  when no spike exists, where $G^{y_n}$ is the M-P distribution with index $y_n$.
And by the Cauchy integral formula,
it can be expressed as $-\frac{1}{2\pi i}\oint_{\gamma_1} f(z)m(z)dz$, where the integral contour $\gamma_1$
is chosen to be positively oriented, enclosing the support of $G^{y_n}$ and it's limit $G^y$. Due to the restriction that $0<y_n<1$, we choose $\gamma_1$ such that the origin $\{z=0\}$ is not enclosed inside.

Using the relationship between $m(z)$ and $\underline{m}(z)$ (the companion Stieltjes transform of $m(z)$): ~$\underline{m}(z)=y_nm(z)-\frac{1-y_n}{z}$, we can rewrite
\begin{eqnarray}\label{1}
G^{y_n}(f)&=&-\frac{1}{2\pi i}\oint_{\gamma_1} f(z)m(z)dz=-\frac{1}{2\pi i}\oint_{\gamma_1} f(z)\Big(\frac{\underline{m}(z)}{y_n}+\frac{1-y_n}{y_nz}\Big)dz\nonumber\\
&=&-\frac{n}{p}\frac{1}{2\pi i}\oint_{\gamma_1} f(z)\underline{m}(z)dz~.
\end{eqnarray}
Besides, for $z \notin supp(G^{y_n})$, $\underline{m}(z)$ satisfies the equation:
\begin{eqnarray}\label{2}
z=-\frac{1}{\underline{m}}+\frac{y_n}{1+\underline{m}}~.
\end{eqnarray}
Taking derivatives on both sides with respect to $z$, we get:
\begin{eqnarray*}
dz=(\frac{1}{\underline{m}^2}-\frac{y_n}{(1+\underline{m})^2})d \underline{m}~.
\end{eqnarray*}
Changing the variable from $z$ to $\underline{m}$ in equation \eqref{1}, we get:
\begin{eqnarray}\label{mp}
G^{y_n}(f)
=-\frac{n}{p}\frac{1}{2\pi i}\oint_{\mathcal{C}_1}f(-\frac{1}{\underline{m}}+\frac{y_n}{1+\underline{m}})
\underline{m}(z)(\frac{1}{\underline{m}^2}-\frac{y_n}{(1+\underline{m})^2})d \underline{m}~.
\end{eqnarray}
Here, the contour $\gamma_1$ of $z$ in equation \eqref{1} is transformed into a contour of $\underline{m}$ through the mapping \eqref{2}, denoted as $\mathcal{C}_1$.

We present the mapping \eqref{2} when $0<y_n<1$ in Figure \ref{graph}, restricting $z$ and $\underline{m}$ to the real \bleu{domain}.  From \citet{SC95}, we know that the $z's$ such that $z^{'}(m)>0$ are not in the support of $G^{y_n}$. Therefore, we shall focus on the increasing intervals, where a one-to-one mapping between $z$ and $\underline{m}$ exists. From the figure, we see that when $\gamma_1$ is chosen to enclose the support of $\reu{G^{y_n}}$: $[a_{y_n}, b_{y_n}]$, the corresponding $\mathcal{C}_1$ will enclose the interval $[\frac{-1}{1-\sqrt{y_n}},\frac{-1}{1+\sqrt{y_n}}]$, and $\underline{m}=-1$ is the pole contained in this interval.
The point on $\gamma_1$ intersecting the real line to the left of $a_{y_n}$ (right of $b_{y_n}$) maps to a point to the left of $\frac{-1}{1-\sqrt{y_n}}$ (right of $\frac{-1}{1+\sqrt{y_n}}$).
Since the imaginary part of $\underline m(z)$ is the same sign as the imaginary part of $z$, we see that $\mathcal{C}_1$ is also oriented counterclockwise.
\begin{figure}[tbp]
\centering
\includegraphics[width=7.5cm]{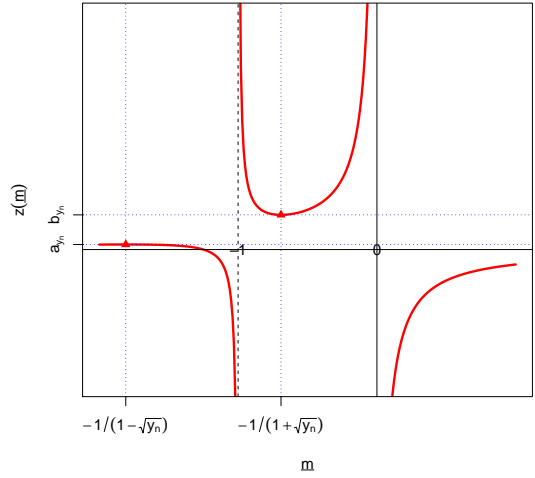}
\caption{The graph of the transform $z(\underline{m})=-\frac{1}{\underline{m}}+\frac{y_n}{1+\underline{m}}$ when $0<y_n<1$. }\label{graph}
\end{figure}

When the spiked structure \eqref{model} exists, by equation
\eqref{MPequation}, this time the companion Stieltjes transform
$\underline{m}=\underline{m}_n$ of \bleu{ $F^{y_n,H_n}$} satisfies
\begin{eqnarray}
&&z=-\frac{1}{\underline{m}}+\frac{p-M}{p}\frac{y_n}{1+\underline{m}}+\frac{y_n}{p}\sum_{i=1}^{k}\frac{a_i n_i}{1+a_i \underline{m}}~,\label{zspike}\label{zmm}\\
&&dz=\bigg(\frac{1}{\underline{m}^2}-\frac{p-M}{p}\frac{y_n}{(1+\underline{m})^2}-\frac{y_n}{p}\sum_{i=1}^{k}\frac{a_i^2 n_i}{(1+a_i \underline{m})^2}\bigg)d \underline{m}~.\nonumber
\end{eqnarray}
Repeating the same computation as before, we get:
\begin{eqnarray}\label{fyh}
&&F^{y_n, H_n}(f)=-\frac{n}{p}\frac{1}{2\pi i}\oint_{\gamma}f(z)\underline{m}(z)dz\nonumber\\
&=&-\frac{n}{p}\frac{1}{2\pi i}\oint_\mathcal{C}f\bigg(-\frac{1}{\underline{m}}+\frac{y_n}{1+\underline{m}}-\frac{y_n}{p}\sum_{i=1}^{k}\frac{(1-a_i)n_i}{(1+\underline{m})(1+a_i \underline{m})}\bigg)
\underline{m}\nonumber\\
&&\times\bigg(\frac{1}{\underline{m}^2}-\frac{y_n}{(1+\underline{m})^2}+\frac{y_n}{p}\sum_{i=1}^{k}n_i\bigg[\frac{1}{(1+\underline{m})^2}-\frac{a_i^2}{(1+a_i \underline{m})^2}\bigg]\bigg)d \underline{m},
\end{eqnarray}
where $\gamma$ \bleu{is}  a positively oriented contour of $z$ that
encloses the support of $F^{S_n}$ and its limit $F^{S}$. From
\citet{Baik06}, we know that under the spiked structure \eqref{model}, the support of $F^{S_n}$
consists of the support of M-P distribution: $[a_{y_n},b_{y_n}]$ plus small intervals
near $\phi(a_i)=a_i+\frac{y_na_i}{a_i-1}$
$(i=1, \cdots, k_1)$. Therefore, the contour $\gamma$ can be expressed
as $\gamma_1\bigoplus (\bigoplus_{i=1}^{k_1} \gamma_{a_i})$
($\gamma_{a_i}$ is denoted as the contour that encloses the point of
$\phi(a_i)$).
\bleu{Moreover}, $\mathcal{C}$ is the image of $\gamma$ under the
mapping \eqref{zmm}, which can also be divided into $\mathcal{C}_1$
plus $\mathcal{C}_{a_i}$ $(i=1, \cdots, k_1)$, with $\mathcal{C}_{a_i}$
enclos\bleu{ing}
$-\frac{1}{a_i}$ and all \bleu{the contours} are non-overlapping and positively oriented.

The term
\begin{eqnarray*}
\frac{y_n}{p}\sum_{i=1}^{k}\frac{(1-a_i)n_i}{(1+\underline{m})(1+a_i \underline{m})}
\end{eqnarray*}
is of order $O(\frac1n)$, so we can take the Taylor expansion of $f$ around the value of $-\frac{1}{\underline{m}}+\frac{y_n}{1+\underline{m}}$, and the term
\begin{eqnarray*}
\frac{y_n}{p}\sum_{i=1}^{k}n_i\bigg[\frac{1}{(1+\underline{m})^2}-\frac{a_i^2}{(1+a_i \underline{m})^2}\bigg]
\end{eqnarray*}
is also of order $O(\frac1n)$.  This gives rise to:
\begin{eqnarray}\label{general}
&&F^{y_n, H_n}(f)
=-\frac{n}{p}\frac{1}{2\pi i}\oint_{\mathcal{C}}f(-\frac{1}{\underline{m}}+\frac{y_n}{1+\underline{m}})(\frac{1}{\underline{m}}-\frac{y_n\underline{m}}{(1+\underline{m})^2})d\underline{m}\nonumber\\
&&-\frac{n}{p}\frac{1}{2\pi i}\oint_{\mathcal{C}}f(-\frac{1}{\underline{m}}+\frac{y_n}{1+\underline{m}})\frac{y_n}{p}\sum_{i=1}^{k}n_i\bigg[\frac{1}{(1+\underline{m})^2}-\frac{a_i^2}{(1+a_i \underline{m})^2}\bigg]\underline{m} d\underline{m}\nonumber\\
&&+\frac{n}{p}\frac{1}{2\pi i}\oint_{\mathcal{C}}f^{'}(-\frac{1}{\underline{m}}+\frac{y_n}{1+\underline{m}})\frac{y_n}{p}\sum_{i=1}^{k}\frac{(1-a_i)n_i}{(1+\underline{m})(1+a_i \underline{m})}(\frac{1}{\underline{m}}-\frac{y_n\underline{m}}{(1+\underline{m})^2})d\underline{m}\nonumber\\
&&+O(\frac{1}{n^2})~.
\end{eqnarray}

Then, we replace $\mathcal{C}$ appearing in equation \eqref{general} by $\mathcal{C}_1\bigoplus(\bigoplus_{i=1}^{k_1}\mathcal{C}_{a_i})$ as mentioned above, and thus we can calculate the value of \eqref{general} separately by calculating the integrals on the contour $\mathcal{C}_{1}$ and each $\mathcal{C}_{a_i}$ $(i=1, \cdots, k_1)$.  If there are no distant spikes then we
will have just \reu{$\mathcal{C}=\mathcal{C}_1$}.

The first term in equation \eqref{general} is equal to
\begin{eqnarray}\label{xin1}
 -\frac{n}{p}\frac{1}{2\pi i}\oint_{\mathcal{C}_1}f(-\frac{1}{\underline{m}}+\frac{y_n}{1+\underline{m}})(\frac{1}{\underline{m}}-\frac{y_n\underline{m}}{(1+\underline{m})^2})d\underline{m}
\end{eqnarray}
for the reason that the only poles: $\underline{m}=0$ and $\underline{m}=-1$ are not enclosed in the contours $C_{a_i}$ $(i=1, \cdots, k_1)$.

Next, we consider these integrals on $\mathcal{C}_{a_i} (i=1, \cdots, k_1)$.

The second term of equation \eqref{general} with the contour being $\mathcal{C}_{a_i}$ is equal to
\begin{eqnarray*}
&&-\frac{n}{p}\frac{1}{2\pi i}\oint_{\mathcal{C}_{a_i}}f(-\frac{1}{\underline{m}}+\frac{y_n}{1+\underline{m}})\frac{y_n}{p}\sum_{i=1}^{k}n_i\bigg[\frac{1}{(1+\underline{m})^2}-\frac{a_i^2}{(1+a_i \underline{m})^2}\bigg]\underline{m} d\underline{m}\\
&=&\frac{n}{p}\frac{1}{2\pi in}\oint_{\mathcal{C}_{a_i}}f(-\frac{1}{\underline{m}}+\frac{y_n}{1+\underline{m}})\sum_{i=1}^{k}\frac{a_i^2n_i \underline{m}}{(1+a_i \underline{m})^2}d\underline{m}\\
&=&\frac{1}{2\pi i p}\oint_{\mathcal{C}_{a_i}}\frac{f(-\frac{1}{\underline{m}}+\frac{y_n}{1+\underline{m}})\underline{m} n_i}{(\underline{m}+\frac{1}{a_i})^2}d\underline{m}\\
&=&\frac{n_i}{p}\Big[f(\phi(a_i))-f^{'}(\phi(a_i))\Big(a_i-\frac{y_na_i}{(a_i-1)^2}\Big)\Big]~,
\end{eqnarray*}
and the third term of equation \eqref{general} with the contour being $\mathcal{C}_{a_i}$ is equal to
\begin{eqnarray*}
&&\frac{n}{p}\frac{1}{2\pi i}\oint_{\mathcal{C}_{a_i}}f^{'}(-\frac{1}{\underline{m}}+\frac{y_n}{1+\underline{m}})\frac{y_n}{p}\sum_{i=1}^{k}\frac{(1-a_i)n_i}{(1+\underline{m})(1+a_i \underline{m})}(\frac{1}{\underline{m}}-\frac{y_n\underline{m}}{(1+\underline{m})^2})d\underline{m}\\
&=&\frac{-1}{2\pi i p}\oint_{\mathcal{C}_{a_i}}f^{'}(-\frac{1}{\underline{m}}+\frac{y_n}{1+\underline{m}})\frac{n_i(1-a_i)}{(\underline{m}+\frac{1}{a_i})a_i(\underline{m}+1)}(\frac{1}{\underline{m}}-\frac{y_n\underline{m}}{(1+\underline{m})^2})d\underline{m}\\
&=&\frac{1}{p}n_if^{'}(\phi(a_i))\Big(a_i-\frac{y_na_i}{(a_i-1)^2}\Big)~.
\end{eqnarray*}
Combining these two terms, we get the influence of the distant spikes, that is, the integral on the contours $\bigcup_{i=1, \cdots, k_1}\mathcal{C}_{a_i}$, which equals to:
\begin{eqnarray}\label{spikevalue}
\frac{1}{p}\sum_{i=1}^{k_1}n_if(\phi(a_i))~.
\end{eqnarray}

So in the remaining part, we only need to consider the integral along the contour $\mathcal{C}_1$.
Consider the second term of \eqref{general} with the contour being $\mathcal{C}_1$:
\begin{eqnarray}\label{onevalue}
&&-\frac{n}{p}\frac{1}{2\pi i}\oint_{\mathcal{C}_1}f(-\frac{1}{\underline{m}}+\frac{y_n}{1+\underline{m}})\frac{y_n}{p}\sum_{i=1}^{k}n_i\bigg[\frac{1}{(1+\underline{m})^2}-\frac{a_i^2}{(1+a_i \underline{m})^2}\bigg]\underline{m} d\underline{m}\nonumber\\
&=&-\frac{1}{2\pi i p}\oint_{\mathcal{C}_1}f(-\frac{1}{\underline{m}}+\frac{y_n}{1+\underline{m}})\bigg[\frac{1}{y_n}(\frac{M \underline{m} y_n}{(1+\underline{m})^2}-\frac{M}{\underline{m}})+\frac{1}{y_n}\frac{M}{\underline{m}}-\sum_{i=1}^{k}\frac{n_i a_i^2\underline{m}}{(1+a_i \underline{m})^2}\bigg]d\underline{m}\nonumber\\
&=&-\frac Mp\frac{n}{2\pi i p}\oint_{\mathcal{C}_1}f(-\frac{1}{\underline{m}}+\frac{y_n}{1+\underline{m}})\big(\frac{\underline{m} y_n}{(1+\underline{m})^2}-\frac{1}{\underline{m}}\big)d\underline{m}\nonumber\\
&&-\frac{1}{2\pi i p}\oint_{\mathcal{C}_1}f(-\frac{1}{\underline{m}}+\frac{y_n}{1+\underline{m}})\bigg(\frac{M}{\underline{m}y_n}-\sum_{i=1}^{k}\frac{n_i a_i^2 \underline{m}}{(1+a_i \underline{m})^2}\bigg)d\underline{m}~.
\end{eqnarray}

Combining  Equations \eqref{mp}, \eqref{xin1}, \eqref{spikevalue} and \eqref{onevalue}, we get:
\begin{eqnarray*}
&&F^{y_n, H_n}(f)=-\frac{1}{2\pi i p}\oint_{\mathcal{C}_1}f(-\frac{1}{\underline{m}}+\frac{y_n}{1+\underline{m}})(\frac{M}{\underline{m}y_n}-\sum_{i=1}^{k}\frac{n_i a_i^2 \underline{m}}{(1+a_i \underline{m})^2})d\underline{m}\\
&&+\frac{1}{2\pi ip}\oint_{\mathcal{C}_1}f^{'}(-\frac{1}{\underline{m}}+\frac{y_n}{1+\underline{m}})\sum_{i=1}^{k}\frac{(1-a_i)n_i}{(1+\underline{m})(1+a_i \underline{m})}(\frac{1}{\underline{m}}-\frac{y_n\underline{m}}{(1+\underline{m})^2})d\underline{m}\\
&&+(1-\frac{M}{p})G^{y_n}(f)+\sum_{i=1}^{k_1}\frac{n_i}{p}f(\phi(a_i))+O(\frac{1}{n^2})~.\nonumber
\end{eqnarray*}

\noindent \reu{\underline{Case of $y_n>1$:}}

We also present the mapping \eqref{2} when $y_n>1$ in Figure \ref{graph2}  below.
\begin{figure}[h]
\centering
\includegraphics[width=7.5cm]{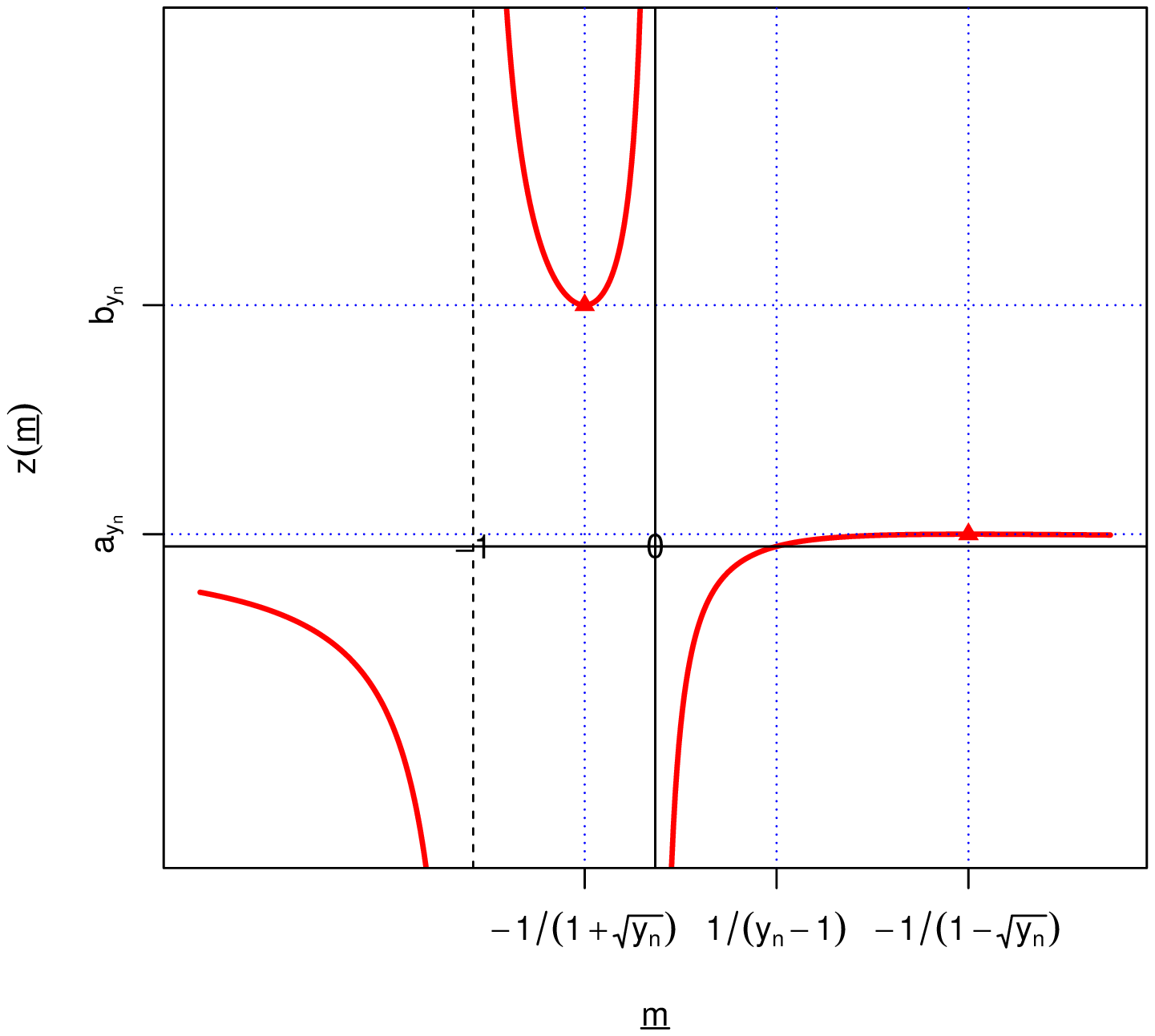}
\caption{The graph of the transform $z(\underline{m})=-\frac{1}{\underline{m}}+\frac{y_n}{1+\underline{m}}$ when $y_n>1$. }\label{graph2}
\end{figure}

When $y_n>1$ there will be mass $1-1/y_n$ at zero.  Assume first that $f$ is analytic on an open interval containing 0 and $b_{y_n}$ and let $\gamma_1$ be a contour covering $[a_{y_n},b_{y_n}]$.
 Then we have \reu{in place  of \eqref{1},}
\begin{eqnarray*}\label{Gb11}
G^{y_n}(f)&=&(1-\frac1{y_n})f(0)-\frac1{2\pi i}\oint_{\gamma_1}f(z)m(z)dz\\
&=&(1-\frac1{y_n})f(0)-\frac1{2\pi iy_n}\oint_{\gamma_1}f(z)\underline m(z)dz.
\end{eqnarray*}
This time the $\underline m$ value corresponding to $a_{y_n}$, namely $\frac{-1}{1-\sqrt {y_n}}$, is positive, and so when changing variables the new contour $\mathcal{C}$ covers
$[c_n,d_n]$ where $c_n<0$ is slightly to the right of $\frac{-1}{1+\sqrt {y_n}}$, and $d_n>0$ is slightly to the left of $\frac{-1}{1-\sqrt {y_n}}$, This interval includes the origin and
not $-1$, and is oriented in a clockwise direction.
We present these two contours $\gamma_1$ and \reu{$\mathcal{C}_1$} in  Figure \ref{zm}.
\begin{figure}[htbp]
\centering
\begin{minipage}[c]{0.5\textwidth}
\centering
\includegraphics[width=6.5cm]{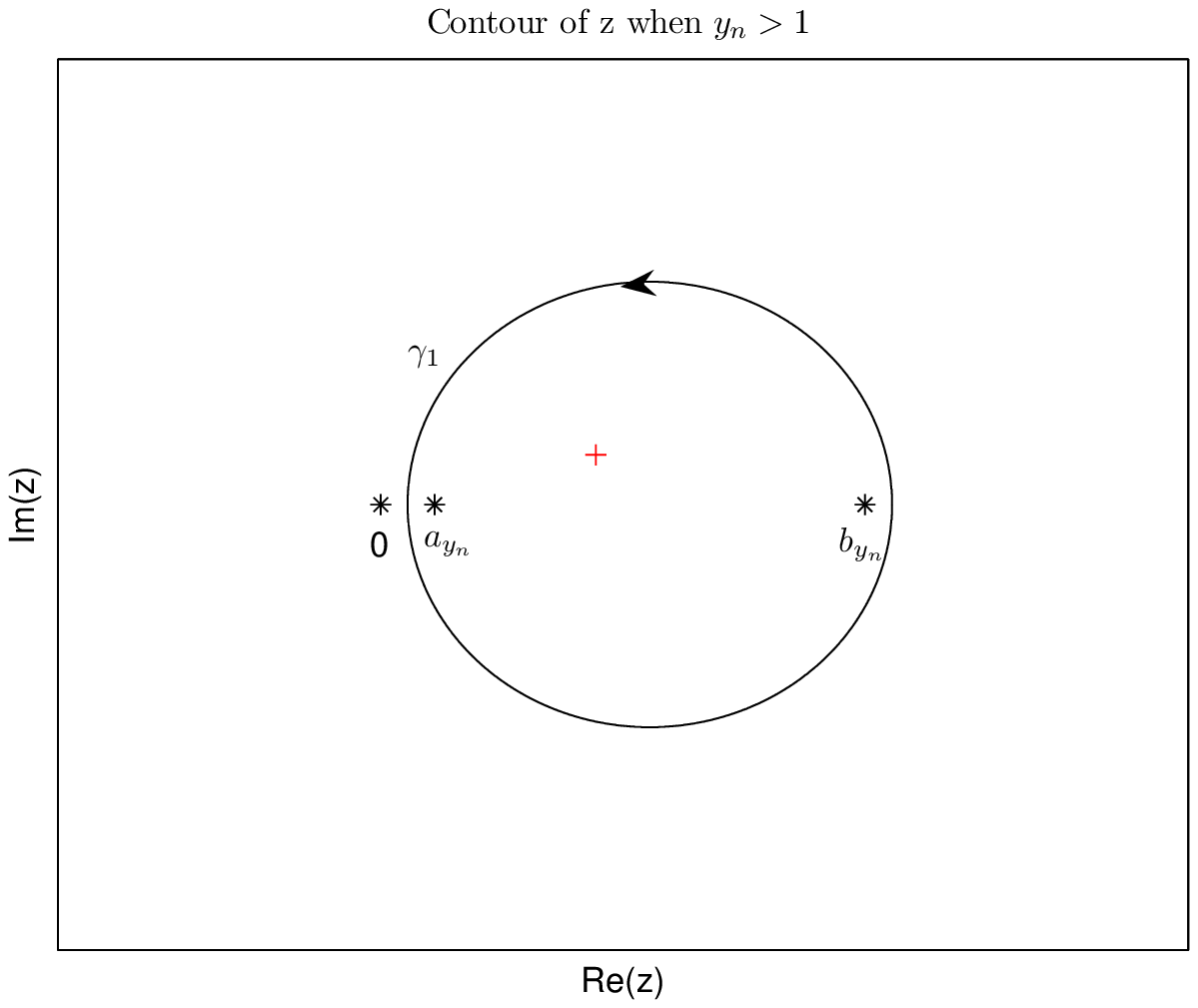}
\end{minipage}%
\begin{minipage}[c]{0.5\textwidth}
\centering
\includegraphics[width=6.5cm]{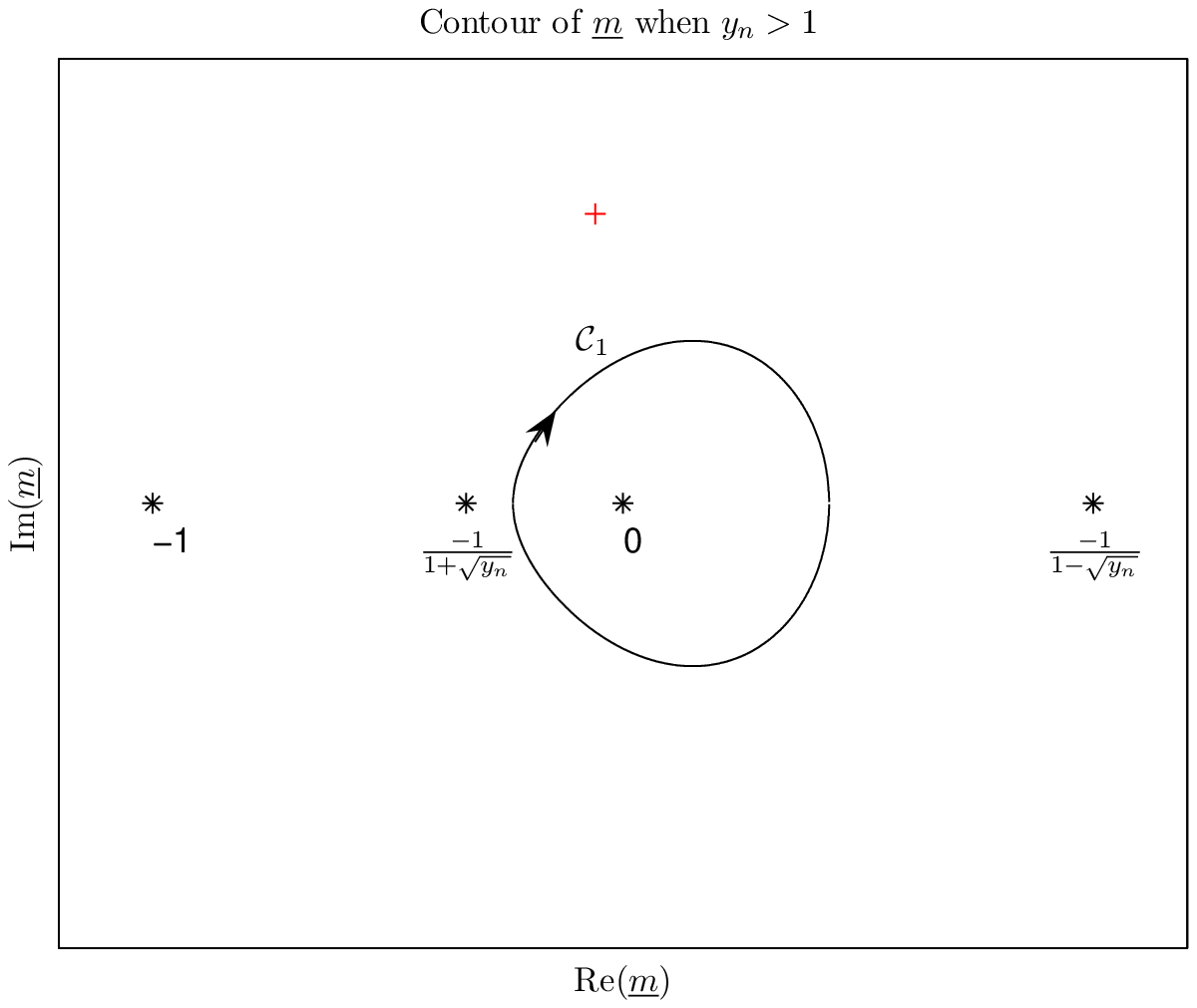}
\end{minipage}
\caption{Contours of $z$ and $\underline{m}$ when $y_n>1$. }\label{zm}
\end{figure}

\noindent We have \reu{in place of \eqref{mp},}
\begin{eqnarray*}\label{Gb12}
G_{y_n}(f)=(1-\frac1{y_n})f(0)-\frac1{2\pi iy_n}\oint_{\reu{\mathcal{C}_1}}f(-\frac1{\underline m}+\frac{y_n}{1+\underline m})\underline m(\frac1{\underline m^2}-\frac{y_n}{(1+\underline m)^2})
d\underline m.
\end{eqnarray*}

Extend \reu{$\mathcal{C}_1$} to the following contour.  On the right side on the real line continue \reu{$\mathcal{C}_1$} to a number large number $r$, then go on a circle $\mathcal{C}(r)$ with radius $r$
in a counterclockwise direction until it returns to the point \reu{$r-i0$}, then go left till it hits \reu{$\mathcal{C}_1$}.  This new contour covers pole $-1$ and not the origin, see Figure \ref{new}.
\begin{figure}[h]
\centering
\includegraphics[width=7.5cm]{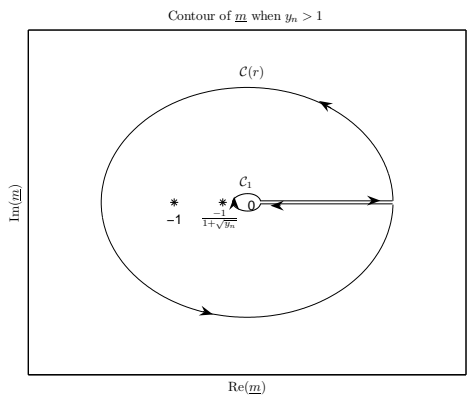}
\caption{The new contour of $\underline{m}$ when $y_n>1$}\label{new}
\end{figure}
 On $\mathcal{C}(r)$
we have using the dominated convergence theorem
\begin{eqnarray*}
 &&\frac1{2\pi iy_n}\oint_{\mathcal{C}(r)}f(-\frac1{\underline m}+\frac{y_n}{1+\underline m})\underline m(\frac1{\underline m^2}-\frac{y_n}{(1+\underline m)^2})
d\underline m\\
 &&(\text{with } \underline m=re^{i\theta})\\
&=&\frac1{2\pi y_n}\int_0^{2\pi}f(-\frac1{\underline m}+\frac{y_n}{1+\underline m})(1-\frac{y_n\underline m^2}{(1+\underline m)^2})
d\theta\\
&&\to \quad \frac{1-y_n}{y_n}\reu{f(0)}\quad(\text{as }r\to\infty).\quad
\end{eqnarray*}
Therefore
\begin{eqnarray}\label{Gb1}
G_{y_n}(f)=-\frac np\frac1{2\pi i}\oint_{\reu{\mathcal{C}_1}}f(-\frac1{\underline m}+\frac{y_n}{1+\underline m})\underline m(\frac1{\underline m^2}-\frac{y_n}{(1+\underline m)^2})
d\underline m.
\end{eqnarray}
where \reu{$\mathcal{C}_1$} just covers \reu{$[-1, \frac{-1}{1+\sqrt {y_n}}]$}.

When there are spikes the only distant ones are those for which $a_i>1+\sqrt {y_n}$.  We will get after the change of variable to $\underline m$ a contour which covers now
$[c'_n,d'_n]$ where $c'_n<0$ is to the right of the largest of $-\frac1{a_i}$ among the distant  spikes (to the right of $\frac{-1}{1+\sqrt {y_n}}$ if there are no
distant spikes), and $d'_n>0$ is to the left of $\frac{-1}{1-\sqrt {y_n}}$, and oriented clockwise.  We can extend the contour as we did before and get the same limit on the
circle as when there are no spikes.  Therefore we get exactly \eqref{fyh} where now the contour $\mathcal{C}$ contains $-1$ and the largest of $-\frac1{a_i}$ among the
distant spikes (contain $\frac{-1}{1+\sqrt {y_n}}$ if there are no distant spikes).
 \reu{Next, we can follow the same proof as for the case $0<y_n<1$, by slitting the contour $\mathcal{C}$ into $\mathcal{C}=\mathcal{C}_1\bigoplus (\bigoplus_{i=1}^{k_1}\mathcal{C}_{a_i})$, where now $\mathcal{C}_1$ just contains the interval $[-1,\frac{-1}{1+\sqrt{y_n}}]$ and the contours $\mathcal{C}_{a_i}$ contain  the influence of $k_1$ distant spikes $a_i>1+\sqrt{y_n}$: $-\frac{1}{a_i}$ ($i=1,\cdots, k_1$), respectively. We thus obtain the same formula as in the case $0<y_n<1$.}
 Therefore Theorem 1 follows where $\mathcal{C}_1$ contains just \reu{$[-1,\frac{-1}{1+\sqrt {y_n}}]$},
and none of the $-\frac1{a_i}$ among the distant spikes ($-\frac{1}{a_i}$ are enclosed in the contour $\mathcal{C}_{a_i}$ as the case of $0<y_n<1$).
\smallskip

\noindent \reu{\underline{Case of $y_n=1$:}}

For $y_n=1$ we have $m(z)=\underline m(z)$, and the contour defining $G_1(f)$ must contain the interval $[0,4]$.  The contour in $\underline{m}$ contains $[c_n,d_n]$ where $-\frac12<c_n<0$,
$d_n>0$ and again is oriented in the clockwise direction.   Extending again this contour we find the limit of the integral on the circle is zero for both $G_1(f)$ and
$F^{1,H_n}(f)$, and we get again Theorem 1 where $\mathcal{C}_1$ is a contour containing \reu{$[-1,-\frac12]$}, and not the origin.

The proof of the theorem is complete.
\end{proof}

%


\section{An application to the test of presence of spike eigenvalues}
\label{application}

In
\citet{JBZ09}, a corrected likelihood ratio statistic
 $\tilde{L}^{*}$ is proposed to test the hypothesis
\[
H_0: ~\Sigma=I_p~~~vs.~~~H_1:~ \Sigma\neq I_p~.
\]
They prove that under $H_0$,
\[
\tilde{L}^{*}-p G^{y_n,H_n}(g)\Rightarrow N(m(g), v(g))~,
\]
where
\begin{eqnarray*}
&&\tilde{L}^{*}=tr S_n-\log|S_n|-p~,\\
&& G^{y_n,H_n}(g)=1-\frac{y_n-1}{y_n}\log(1-y_n)~,\\
&&m(g)=-\frac{\log(1-y)}{2}~,\\
&&v(g)=-2\log(1-y)-2y.
\end{eqnarray*}
At a significance level $\alpha$ (usually 0.05), the test will reject
$H_0$ when
$\tilde{L}^{*}-p G^{y_n,H_n}(g) > m(g) + \Phi^{-1}(1-\alpha)\sqrt{v(g)}$
where $\Phi$ is the standard normal cumulative distribution function.

However, the power function of this test remains unknown because the
distribution of  $\tilde{L}^{*}$ under the general alternative
hypothesis
$H_1$  is ill-defined.
Let's consider this general test as a way to test the null hypothese
$H_0$ above against an alternative hypothesis of the form:
\[ H_1^*:~ \Sigma_p \textrm{ has the spiked structure \protect\eqref{model}.}
\]
In other words, we want to test the absence against the presence of
possible spike eigenvalues in the population covariance matrix.
The general asymptotic expansion in Theorem~\ref{formula}
helps to find the power function of the test.

More precisely, under the alternative $H_1^*$ and
for $f(x)=x-\log x-1$ used in the
statistic  $\tilde{L}^{*}$, the centering  term
$F^{y_n,H_n}(f)$ can be found to be
\[
1+\frac{1}{p}\sum_{i=1}^{k}n_ia_i-\frac{M}{p}-\frac{1}{p}\sum_{i=1}^{k}n_i\log a_i-(1-\frac{1}{y_n})\log(1-y_n)+O(\frac{1}{n^2})~,
\]
thanks to the following formulas
\begin{eqnarray}\label{app1}
  F^{y_n, H_n}(x)=1+\frac{1}{p}\sum_{i=1}^{k}n_i a_i-\frac{M}{p}+O(\frac{1}{n^2})~
\end{eqnarray}
and
\begin{eqnarray}\label{app2}
F^{y_n, H_n}(\log x)=\frac{1}{p}\sum_{i=1}^{k}n_i\log a_i-1+(1-\frac{1}{y_n})\log(1-y_n)+O(\frac{1}{n^2})~.
\end{eqnarray}
The details of derivation of these formulas are given  in the Appendix \ref{aa}.

Therefore we have obtained that under $H_1^{*}$,
\[
\tilde{L}^{*}-pF^{y_n, H_n} (f) \Rightarrow N(m(g), v(g))~.
\]
It follows that
the asymptotic power function of the test is
\[
\beta(\alpha)=1-\Phi\Bigg(\Phi^{-1}(1-\alpha)-\frac{\sum_{i=1}^{k}n_i(a_i-1-\log a_i)}{\sqrt{-2\log(1-y)-2y}}\Bigg)~.
\]
In the particular case where the spiked model has
only one simple close spike, i.e. $k=1$, $k_1=0$, $n_1=1$,
the above power function becomes
\[
\beta(\alpha)=1-\Phi\left(\Phi^{-1}(1-\alpha)-\frac{a_1-1-\log a_1}{\sqrt{-2\log(1-y)-2y}}\right)~,
\]
which is exactly the formula $(5.6)$ found  in \citet{Onatski11}.
Note that these authors have found this formula using a sophisticated
tools  of asymptotic contiguity and Le Cam's first and third lemmas,
our derivation is in a sense much more direct.

\appendix

\section{Additional proofs of $\eqref{app1}$ and $\eqref{app2}$}\label{aa}
\reu{The likelihood ratio test works only when $0<y_n<1$, and when $k_1+1\leq i\leq k$, the corresponding $a_i$ satisfy  $|a_i-1|\leq \sqrt{y_n}$, which is equivalent to  $-\frac{1}{a_i} \in [\frac{-1}{1-\sqrt{y_n}},\frac{-1}{1+\sqrt{y_n}}]$, so poles of $\{\underline{m}=-1\}$, $\{\underline{m}=-\frac{1}{a_i},i=(k_1+1,\cdots,k)\}$  and $\{\underline{m}=\frac{1}{y_n-1}\}$ (pole of the function $log z$) should be included in $\mathcal{C}_1$. In all the following, we write $m$ to stand for $\underline{m}$ for convenience.}
\subsection{Proof of $\eqref{app1}$}
We have
\begin{eqnarray}\label{x1}
\eqref{th1}=-\frac{1}{2\pi i p}\oint_{\mathcal{C}_1}(-\frac{1}{m}+\frac{y_n}{1+m})(\frac{M}{y_n m}-\sum_{i=1}^k \frac{n_i a_i^2 m}{(1+a_i m)^2})dm~,
\end{eqnarray}
and its residual at $m=-1$ equals to
\begin{eqnarray}\label{nn1}
\frac{M}{p}-\frac{y_n}{p}\sum_{i=1}^{k}\frac{n_i a_i^2}{(1-a_i)^2}~.
\end{eqnarray}
\begin{eqnarray}\label{x2}
\eqref{th2}=\frac{1}{2 \pi i p}\oint_{\mathcal{C}_1}\sum_{i=1}^{k}\frac{(1-a_i)n_i}{(1+a_i m)(1+m)}\bigg(\frac{1}{m}-\frac{y_n m}{(1+m)^2}\bigg)dm~,
\end{eqnarray}
and its residual at $m=-1$ equals to
\begin{eqnarray}\label{nn2}
&&\frac{1}{p}\sum_{i=1}^{k}\bigg[-n_i-\frac12(1-a_i)n_i y_n\frac{\partial}{\partial m^2}\Big(\frac{m}{1+a_i m}\Big)^2\Big|_{m=-1}\bigg]\nonumber\\
&=&\frac{1}{p}\sum_{i=1}^{k}\bigg[-n_i+\frac{a_i n_i y_n}{(1-a_i)^2}\bigg]~.
\end{eqnarray}
Besides, the residual of \eqref{x1}+\eqref{x2} at $m=-\frac{1}{a_i}$, $i=(k_1+1,\cdots, k)$ can be calculated as
\begin{eqnarray}\label{nn3}
\frac{1}{p}n_i\bigg(a_i+\frac{y_n a_i}{a_i-1}\bigg)~.
\end{eqnarray}
\begin{eqnarray}\label{x3}
\eqref{th3}&=&1-\frac{M}{p}+\frac{1}{p}\sum_{i=1}^{k_1}n_i\bigg(a_i+\frac{y_n a_i}{a_i-1}\bigg)+O(\frac{1}{n^2})~.
\end{eqnarray}
Combine \eqref{nn1}, \eqref{nn2}, \eqref{nn3} and \eqref{x3}, we get:
\begin{eqnarray*}
F^{y_n, H_n}(x)=1+\frac{1}{p}\sum_{i=1}^{k}n_i a_i-\frac{M}{p}+O(\frac{1}{n^2})~.
\end{eqnarray*}

\subsection{Proof of $\eqref{app2}$}
We first calculate \eqref{th1} and \eqref{th2} by considering their residuals at $m=-1$.
\begin{eqnarray}\label{11}
\eqref{th1}&=&\frac{-1}{2\pi i p y_n}\oint_{\mathcal{C}_1}\frac{\log(\frac{y_n-1}{m})+\log(\frac{m-\frac{1}{y_n-1}}{m+1})}{m}(M-\sum_{i=1}^{k}\frac{n_i a_i^2y_n m^2}{(1+a_im)^2})dm\nonumber\\
&=&\frac{-M}{2\pi i p y_n}\oint_{\mathcal{C}_1}\log(\frac{m-\frac{1}{y_n-1}}{m+1})\frac{1}{m}dm\nonumber\\
&&+
   \frac{1}{2\pi i p y_n}\oint_{\mathcal{C}_1}\log(\frac{m-\frac{1}{y_n-1}}{m+1})\sum_{i=1}^{k}\frac{n_i a_i^2 y_n m}{(1+a_i m)^2}dm\nonumber\\
&\triangleq& A+B~.
\end{eqnarray}
\begin{eqnarray}\label{a}
A&=&\frac{-M}{2\pi i p y_n}\oint_{\mathcal{C}_1}\log(\frac{m-\frac{1}{y_n-1}}{m+1})\cdot d \log m\nonumber\\
&=&\frac{M}{2\pi i p y_n}\oint_{\mathcal{C}_1}\log m \cdot d\log(\frac{m-\frac{1}{y_n-1}}{m+1})\nonumber\\
&=&\frac{M}{2\pi i p y_n}\cdot \frac{y_n}{y_n-1}\oint_{\mathcal{C}_1}\frac{\log m}{(m+1)(m-\frac{1}{y_n-1})}dm\nonumber\\
&=&-\frac{M}{p y_n}\log(1-y_n)~,
\end{eqnarray}
\begin{eqnarray}\label{b}
B&=&\frac{1}{2\pi i p}\oint_{\mathcal{C}_1}\log(\frac{m-\frac{1}{y_n-1}}{m+1})\sum_{i=1}^{k}\frac{n_i a_i^2 m}{(1+a_i m)^2}dm\nonumber\\
&=&\frac{1}{2\pi i p}\sum_{i=1}^{k}\oint_{\mathcal{C}_1}\log(\frac{m-\frac{1}{y_n-1}}{m+1})n_i a_i(\frac{1}{1+a_i m}-\frac{1}{(1+a_i m)^2})dm\nonumber\\
&\triangleq&C-D~,
\end{eqnarray}
where
\begin{eqnarray}\label{c}
C&=&\frac{1}{2\pi i p}\sum_{i=1}^{k}\oint_{\mathcal{C}_1}\log(\frac{m-\frac{1}{y_n-1}}{m+1})\frac{n_i a_i}{1+a_i m}dm\nonumber\\
&=&\frac{1}{2\pi i p}\sum_{i=1}^{k}\oint_{\mathcal{C}_1}n_i\log(\frac{m-\frac{1}{y_n-1}}{m+1})\cdot d\log({1+a_i m})\nonumber\\
&=&\frac{-1}{2\pi i p}\sum_{i=1}^{k}\oint_{\mathcal{C}_1}n_i\log{(1+a_i m)}\cdot d\log(\frac{m-\frac{1}{y_n-1}}{m+1})\nonumber\\
&=&\frac{-1}{2\pi i p}\cdot\frac{y_n}{y_n-1}\sum_{i=1}^{k}\oint_{\mathcal{C}_1}\frac{n_i\log{(1+a_i m)}}{(m+1)(m-\frac{1}{y_n-1})}dm\nonumber\\
&=&\frac1p\sum_{i=1}^{k}n_i\log(1-a_i)-\frac1p\sum_{i=1}^{k}n_i\log(1+\frac{a_i}{y_n-1})~,
\end{eqnarray}
and
\begin{eqnarray}\label{d}
D&=&\frac{1}{2\pi ip}\sum_{i=1}^{k}\oint_{\mathcal{C}_1}\log(\frac{m-\frac{1}{y_n-1}}{m+1})\frac{n_i a_i}{(1+a_i m)^2}dm\nonumber\\
&=&\frac{1}{2\pi ip}\sum_{i=1}^{k}\oint_{\mathcal{C}_1}\frac{n_i}{1+a_i m}\cdot d\log(\frac{m-\frac{1}{y_n-1}}{m+1})\nonumber\\
&=&\frac{y_n}{2\pi i p(y_n-1)}\sum_{i=1}^{k}\oint_{\mathcal{C}_1}\frac{n_i}{(1+a_i m)(m-\frac{1}{y_n-1})(m+1)}dm\nonumber\\
&=&\frac{1}{p}\sum_{i=1}^{k}(\frac{n_i}{1+\frac{a_i}{y_n-1}}-\frac{n_i}{1-a_i})~.
\end{eqnarray}
Combine \eqref{11}, \eqref{a}, \eqref{b}, \eqref{c} and \eqref{d}, we get the residual of \eqref{th1} at $m=-1$:
\begin{eqnarray}\label{12}
&&-\frac{M}{p y_n}\log(1-y_n)+\frac{1}{p}\sum_{i=1}^{k}n_i\log(1-a_i)-\frac{1}{p}\sum_{i=1}^{k}n_i\log(1+\frac{a_i}{y_n-1})\nonumber\\
&&-\frac{1}{p}\sum_{i=1}^{k}\frac{n_i}{1+\frac{a_i}{y_n-1}}+\frac{1}{p}\sum_{i=1}^{k}\frac{n_i}{1-a_i}~.
\end{eqnarray}
Then, we consider the part \eqref{th2} in the general formula influenced by the pole $m=-1$:
\begin{eqnarray*}
\eqref{th2}&=&-\frac{1}{2\pi i p}\oint_{\mathcal{C}_1}f^{'}(-\frac{1}{m}+\frac{y_n}{1+m})\sum_{i=1}^{k}(\frac{n_i a_i}{1+a_i m}-\frac{n_i}{1+m})(\frac{1}{m}-\frac{y_n m}{(1+m)^2})dm\\
&=&-\frac{1}{2\pi i p}\sum_{i=1}^{k}n_i\oint_{\mathcal{C}_1}\frac{m(m+1)}{y_n m-m-1}(\frac{a_i}{1+a_i m}-\frac{1}{1+m})(\frac{1}{m}-\frac{y_n m}{(1+m)^2})dm\\
&\triangleq &\frac{-1}{2\pi ip(y_n-1)}\sum_{i=1}^{k}n_i(E-F-G+H)~,
\end{eqnarray*}
where
\begin{eqnarray*}
&&E=\oint_{\mathcal{C}_1}\frac{a_i(m+1)}{(1+a_i m)(m-\frac{1}{y_n-1})}=2\pi i\frac{y_n a_i}{y_n+a_i-1}~,\\
&&F=\oint_{\mathcal{C}_1}\frac{a_i y_nm^2}{(m+1)(1+a_i m)(m-\frac{1}{y_n-1})}=2\pi i(\frac{a_i(y_n-1)}{a_i-1}+\frac{a_i}{y_n+a_i-1})~,\\
&&G=\oint_{\mathcal{C}_1}\frac{1}{m-\frac{1}{y_n-1}}=2\pi i~,\\
&&H=\oint_{\mathcal{C}_1}\frac{y_n m^2}{(m+1)^2(m-\frac{1}{y_n-1})}dm=2\pi iy_n~.
\end{eqnarray*}
Collecting these four terms, we have the residual of \eqref{th2} at $m=-1$:
\begin{eqnarray}\label{22}
\frac{1}{p}\sum_{i=1}^{k}(\frac{1}{a_i-1}-\frac{a_i}{y_n+a_i-1})n_i~.
\end{eqnarray}
Then we consider the influence of \eqref{th1}+\eqref{th2} caused by the pole $m=-\frac{1}{a_i}$, $i=k_1+1, \cdots, k$, which can be calculated similarly as
\begin{eqnarray}\label{nn4}
\frac{n_i}{p}\log(a_i+\frac{y_na_i}{a_i-1})~.
\end{eqnarray}
Finally, using the known result that $G^{y_n}(\log x)=(1-\frac{1}{y_n})\log(1-y_n)-1$, which has been calculated in \citet{BS04}, and combine \eqref{12}, \eqref{22}, \eqref{nn4} and \eqref{th3}, we get
\begin{eqnarray*}
F^{y_n, H_n}(\log x)=\frac{1}{p}\sum_{i=1}^{k}n_i\log a_i-1+(1-\frac{1}{y_n})\log(1-y_n)+O(\frac{1}{n^2})~.
\end{eqnarray*}


\begin{thebibliography}{99}

\bibitem[{Bai and Silverstein(2004)}]{BS04}
  Bai, Z.D. and Silverstein, J.W. (2004).
  {CLT} for linear spectral statistics of large-dimensional sample
  covariance matrices.
  {\em Ann. Probab.}, {\bf 32}, 553--605.


\bibitem[{Bai and Yao(2008)}]{BaiYao08}
  Bai, Z.D. and Yao, J.F. (2008).
  {CLT} for eigenvalues in a spiked population model.
  {\em Ann. Inst. Henri Poincar\'{e} Probab. Stat.}, {\bf 44}(3), 447--474.

\bibitem[{Bai {et~al.} (2009)}]{JBZ09}
 Bai, Z.D.,  Jiang, D.D., Yao, J.F. and Zheng, S.R. (2009).
 Corrections to LRT on large dimensional covariance matrix
 by RMT.
 {\em Ann. Statist.} {\bf 37}, 3822-3840.

\bibitem[Bai and Silverstein(2010)]{BSbook}
  Bai, Z.D. and Silverstein, J.W. (2010).
\newblock \emph{Spectral Analysis of Large Dimensional Random
  Matrices} (2nd edition).
\newblock Springer, 20.

\bibitem[{Bai and Yao(2012)}]{BY12}
  Bai, Z.D. and Yao, J.F. (2012).
  On sample eigenvalues in a generalized spiked population model.
  {\em J. Multivariate Anal.}, {\bf 106}, 167--177.

\bibitem[{Baik et~al.(2005)}]{BBP05}
 Baik, J., Ben Arous, G. and Péché, S. (2005).
\newblock Phase transition of the largest eigenvalue for nonnull complex sample
  covariance matrices.
\newblock \emph{Ann. Probab.}, {\bf 33}(5),  1643--1697.


\bibitem[{Baik and Silverstein(2006)}]{Baik06}
  Baik, J. and Silverstein, J.W. (2006).
  Eigenvalues of large sample covariance matrices of spiked population models.
  {\em J. Multivariate Anal.}, {\bf 97}, 1382--1408.



\bibitem[{Benaych-Georges et~al. (2011)}]{Benaych11}
  Benaych-Georges, F., Guionnet, A. and Maida, M. (2011).
  Fluctuations of the extreme eigenvalues of finite rank deformations of random matrices.
  {\em Electron. J. Probab.}, {\bf 16}, 1621--1662.

\bibitem[{Benaych-Georges and Nadakuditi(2011)}]{BenaychNadakuditi11}
  Benaych-Georges, F. and Nadakuditi, R.R. (2011).
  The eigenvalues and eigenvectors of finite low rank perturbations of large random matrices.
  {\em Adv. Math.}, {\bf 227}(2), 494--521.




\bibitem[{Johnstone(2001)}]{John01}
  Johnstone, I.M. (2001).
  On the distribution of the largest eigenvalue in principal components analysis.
  {\em Ann. Statist.}, {\bf 29}(2), 295--327.

\bibitem[{Kritchman and Nadler(2008)}]{KN08}
Kritchman, S. and  Nadler, B. (2008).
Determining the number of components in a factor model
from limited noisy data.
{\em Chem. Int. Lab. Syst.} {\bf 94}, 19-32.

\bibitem[{Kritchman and Nadler(2009)}]{KN09}
Kritchman, S. and  Nadler, B. (2009).
Non-parametric detection of the number of signals:
Hypothesis testing and random matrix theory.
{\em IEEE Trans. Signal Process.} {\bf  57}(10),
3930–3941.


\bibitem[Mar\v{c}enko and Pastur(1967)]{MP}
Mar\v{c}enko, V.A. and Pastur, L.A. (1967).
\newblock Distribution of eigenvalues for some sets of random matrices.
\newblock \emph{Math. USSR-Sb}, {\bf 1},  457--483.

\bibitem[Nadakuditi and Silverstein(2010)]{NS}
Nadakuditi, R.R. and Silverstein, J.W. (2010).
Fundamental limit of sample generalized eigenvalue based detection of
signals in noise
using relatively few signal-bearing and noise-only samples.
{\em IEEE J. Sel. Topics Signal Processing.} {\bf 4}(3), 468-480.

\bibitem[Onatski(2009)]{Onatski09}
  Onatski, A. (2009).
  Testing hypotheses about the number of
  factors in large factor models.
  {\em  Econometrica}, {\bf 77}(5), 1447-1479.


\bibitem[Onatski(2010)]{Onatski10}
  Onatski, A. (2010).
  Determining the number of factors from empirical distribution of
  eigenvalues.
  {\em Review of Economics and Statistics},
  {\bf 92}(4), 1004-1016.

\bibitem[{Onatski et al. (2011)}]{Onatski11}
  Onatski, A., Moreira, M.J. and Hallin, M. (2011).
  Asymptotic power of sphericity tests for high-dimensional data.
  {\em Preprint}, available at
  \texttt{arXiv:1210.5663v1}.

\bibitem[Onatski(2012)]{Onatski12}
 Onatski, A. (2012).
 \newblock
 Asymptotics of the principal components estimator of large factor
 models with weakly influential factors.
 {\em J. Econometrics}, {\bf 168}, 244-258.

\bibitem[{Pan and Zhou(2008)}]{PanZhou08}
  Pan, G.M. and Zhou, W. (2008).
  Central limit theorem for signal-to-interference ratio of reduced
  rank linear receiver.
  {\em Ann. Appl. Probab.}, {\bf 18}, 1232-1270.


\bibitem[{Passemier and Yao (2012)}]{PY12}
 Passemier, D. and Yao, J.F. (2012).
 On determining the number of spikes
 in a high-dimensional spiked population model.
 {\em Random Matrix:
   Theory and Applciations}, {\bf 1},
 1150002


\bibitem[{Paul(2007)}]{Paul07}
  Paul, D. (2007).
  Asymptotics of sample eigenstructure for a large dimensional spiked covariance model.
  {\em Statist. Sinica.}, {\bf 17}, 1617--1642.

\bibitem[Silverstein(1995)]{Silverstein95}
 Silverstein, J.W. (1995).
\newblock Strong convergence of the empirical distribution of eigenvalues of
  large-dimensional random matrices.
\newblock \emph{J. Multivariate Anal.}, {\bf 55}\penalty0 (2),~\penalty0 331--339.

\bibitem[Silverstein and Choi(1995)]{SC95}
 Silverstein, J.W. and Choi, S.I. (1995).
\newblock Analysis of the limiting spectral distribution of large dimensional random matrices.
\newblock \emph{J. Multivariate Anal.}, {\bf54}\penalty0 (2),~\penalty0 295--309.
\end{thebibliography}
\end{document}